\documentclass[reqno,12pt]{amsart}
\usepackage{amsmath,amsthm,amssymb}
\usepackage{tikz}
\usetikzlibrary{arrows}
\usetikzlibrary{backgrounds, calc, positioning}

\date{}
\newcommand{\RR}{\mathbb{R}}

\newtheorem{theorem}{Theorem}[section]
\newtheorem{lemma}[theorem]{Lemma}
\newtheorem{corollary}[theorem]{Corollary}
\newtheorem{proposition}[theorem]{Proposition}

\newtheorem{remark}[theorem]{Remark}

\numberwithin{equation}{section}

\usepackage{tikz}
\usetikzlibrary{arrows}
\usepackage{tikz}
\usetikzlibrary{shapes,snakes}
\usetikzlibrary{backgrounds, calc, positioning}

\begin{document}

\title{Remarks on curvature in the transportation metric}

\author{Bo'az~Klartag}
\address{School of Mathematical Sciences, Tel Aviv University, Tel Aviv 69978, Israel}
\email{klartagb@tau.ac.il}

\vspace{5mm}
\author{Alexander~V.~Kolesnikov}
\address{ Higher School of Economics, Moscow,  Russia}
\email{Sascha77@mail.ru}

\thanks{
The first named author was supported by a grant from the European Research Council.
The second named author was supported by the
 DFG project RO 1195/12-1.
The article was prepared within the framework of the Academic Fund Program at the National Research University Higher School of Economics (HSE) in 2017--2018 (grant No 17-01-0102) and by the Russian Academic Excellence Project ''5-100''.
}


\begin{abstract}
According to a classical result of E.~Calabi any hyperbolic affine hypersphere endowed with its natural Hessian metric has a non-positive Ricci tensor.
The affine hyperspheres can be described as the level sets of solutions to the ``hyperbolic" toric K\"ahler-Einstein equation
$e^{\Phi} = \det D^2 \Phi$ on  proper convex cones. We prove a generalization of this theorem by showing that for every
 $\Phi$  solving  this equation on a proper convex domain $\Omega$ the corresponding metric measure space
$(D^2 \Phi, e^{\Phi}dx)$ has a non-positive Bakry-{\'E}mery tensor.
Modifying the Calabi computations we obtain this result by applying the tensorial maximum principle to the weighted Laplacian of the Bakry-{\'E}mery tensor. Our computations are carried out in a generalized framework adapted to the optimal transportation problem for arbitrary target and source measures.
For the optimal transportation of the log-concave probability measures we prove a third-order  uniform dimension-free apriori estimate
in the spirit of the second-order Caffarelli contraction theorem, which has numerous applications in probability theory.
\end{abstract}

\maketitle

\section{Introduction}

The toric K\"ahler-Einstein  equation
\begin{equation}
\label{HKE}
e^{-\alpha \Phi} = \det D^2 \Phi
\end{equation}
for a convex function $\Phi: \RR^n \rightarrow \RR$ is a real analog of the complex Monge-Amp{\`e}re equation,
 which is
instrumental in the theory of  K\"ahlerian manifolds. Here, $D^2 \Phi$ is the Hessian of the function $\Phi$.
The equation (\ref{HKE}) is also connected
to convex geometry.
According to the standard classification, we distinguish
the parabolic case $\alpha=0$, elliptic case $\alpha>0$, and hyperbolic case
$\alpha<0$. The reader should not be  confused by the fact that according to the standard PDE terminology,
equation (\ref{HKE}) is {\it always} (non-uniformly) elliptic.

Having in mind potential applications in analysis and probability we deal with a Monge-Amp{\`e}re
equation of a more general type
\begin{equation}
\label{MA}
e^{-V} = e^{-W(\nabla \Phi)} \det D^2 \Phi.
\end{equation}
When the functions $e^{-V}, e^{-W}$ are densities of finite (probability) measures, the equation (\ref{MA})
is  related to the optimal transportation problem (\cite{BoKo2012}, \cite{Vill})
and can be analyzed using  various functional-analytic and measure-theoretic methods.
The measures $\mu = e^{-V} \ dx$, $\nu=e^{-W} dx$ are called source and target measures respectively.
However, here we are interested in situations when the associated measures are not always finite.
An example is given by the hyperbolic case of ({\ref{HKE}), i.e. $\alpha <0$, where  $\mu=e^{-\alpha \Phi} dx$ and
$\nu$ is the Lebesgue measure  on the (typically unbounded) set $\nabla \Phi(\mathbb{R}^n)$.

In this paper we extend the celebrated Calabi approach for affine spheres and compute
the appropriate weighted Laplacian of various tensors.
It turns out that these tensors can be estimated with the help of the maximum principle.
These computations, yet elementary, are tedious and admit a non-trivial geometrical  interpretation.
We try to present them in the simplest but general form, sometimes using associated diagrams.
In our opinion one of the most natural objects to study is the so-called
metric measure space
$$
(h=D^2 \Phi,\mu),
 $$ i.e. the space $\mathbb{R}^n$ (or a subset of it) equipped with
the Hessian  metric $$h= \sum_{i,j=1}^n \Phi_{ij} dx^i dx^j$$ and the mesure $\mu= e^{-V} dx$.
The reader can find explanations justifying this viewpoint and motivating applications in
\cite{Fox},  \cite{K_moment}, \cite{K_part_I}, \cite{KlarKol}, \cite{kol}.
The corresponding second-order differential operator $L$
is the weighted Laplacian
$$
L = \Delta_h - \nabla_h P \cdot \nabla_h,
$$
where $\nabla_h$ is the Riemannian gradient, $\Delta_h$ is the Riemannian (Laplace-Beltrami) Laplacian, and  $P$ is the potential of $\mu$ with respect to the Riemannian volume:
$$e^{-V} = e^{-P} \sqrt{\det D^2 \Phi}.$$

Given a tensor $T$ of any type one can always compute
the corresponding weighted Laplacian
\begin{equation}
L T = \Phi^{pq} \Bigl( \nabla_p \nabla_q T - \nabla_p P \nabla_q  T \Bigr). \label{eq_new}
\end{equation}
Our  central technical result is an exact expression of $L T$
for several important tensors $T$
and  arbitrary measures $\mu$ and $\nu$.
Our computations imply, in particular, that for $\Phi$ solving (\ref{HKE}) the following differential inequality holds:
\begin{equation}\label{Lric}
L \bigl(\mbox{\rm{Ric}}_{\mu} \bigr)   \ge  4 \mbox{\rm{Ric}}_{\mu} \odot g,
 \end{equation}
where $\odot$ is  the symmetric product (see below), where
$$\mbox{\rm{Ric}}_{\mu} = \mbox{\rm{Ric}} + \nabla^2_h P$$
is the Bakry-{\'E}mery tensor of $(h,\mu)$, $\mbox{\rm{Ric}}$
is the standard Ricci tensor of $h$,
and
$$
g_{ij} = \Phi_{iab} \Phi_{j}^{ab}.
$$
The latter tensor was introduced by Calabi in \cite{Calabi1}. It consitutes the most substantial (non-negative) part of the Bakry-{\'E}mery tensor.
Applying the maximum principle we obtain from (\ref{Lric})
that the largest eigenvalue of $\mbox{\rm{Ric}}_{\mu}$ is non-positive at its  maximum point. Therefore:

 \begin{theorem}
 \label{mainth}
 Consider a proper, open convex domain $\Omega \subset \mathbb{R}^n$. Let $\alpha$ be a negative number and  $\Phi$ be the (unique) solution to (\ref{HKE})  on $\Omega$ satisfying
 $\lim_{x \to \partial \Omega} \Phi(x) = + \infty$. Then
 $$
 \mbox{\rm{Ric}}_{\mu} \le 0.
 $$
 \end{theorem}

 In the case when $\Omega$ is a proper convex cone, the unique solution $\Phi$ to (\ref{HKE}) must be logarithmically-homogeneous, i.e.,
 $\Phi(\lambda x) = \Phi(x) + 2 (n / \alpha) \log \lambda$ for any $\lambda > 0$.
 It can be easily verified (see Lemma 5.2) that every logarithmically-homogeneous $\Phi = 2P$  satisfies
 $$
 \nabla^2_h \Phi =  0.
 $$
 In particular, the Ricci and Bakry-{\'E}mery tensors coincide in this case and
  Theorem \ref{mainth}
immediately implies

\begin{corollary} (Calabi \cite{Calabi1}, see also Fox \cite{Fox}, Loftin \cite{Lof} and Sasaki \cite{Sas}).
\label{sph}
 Consider a proper open convex cone $\Omega \subset \mathbb{R}^n$. Let $\alpha$ be a negative number and  $\Phi$ be the (unique) solution to (\ref{HKE}) on $\Omega$   satisfying
 $\lim_{x \to \partial \Omega} \Phi(x) = + \infty$. Then
 $$
 \mbox{\rm{Ric}} \le 0.
 $$
 Moreover, the Ricci tensors of the level sets of $\Phi$ (hyperbolic affine spheres) endowed with the metric $h$ are non-positive.
\end{corollary}

Using the same approach we prove a third-order analog of the so-called contraction theorem (L.~Caffarelli).
According to this theorem every (see more general statement in Section 5) optimal transportation mapping $\nabla \Phi$
pushing forward the standard Gaussian measure $\mu = \frac{1}{(2 \pi)^{\frac{n}{2}}} e^{-\frac{|x|^2}{2}} dx$ onto
a probability measure $\nu = e^{-W} dx$ is a contraction provided $D^2 W \ge \mbox{\rm{Id}}$.
This result is important in probability theory because it implies that the isoperimetric properties of $\mu$
are at least as strong as the isoperimetric properties  of $\gamma$.
Contractivity of the optimal mappings
corresponds to the following uniform estimate
$$
D^2 \Phi(x) \le \mbox{\rm{Id}}, \ \forall x \in  \mathbb{R}^n.
$$
We prove that a similar estimate (global, dimension free, uniform) for third-order derivatives of $\Phi$
holds under certain  natural assumptions on the second and third-order derivatives of $V$ and $W$.

     \section{Notations, definitions, and previously known results}
\label{sec2}
It will be assumed throughout that we are given a smooth (at least $C^5$) convex function $\Phi$ on $\mathbb{R}^n$
such that its gradient $\nabla \Phi$
pushes forward the measure
$$
\mu = e^{-V} dx
$$
onto the measure
$$
\nu = e^{-W} dx.
$$
The potentials $V, W$ are assumed to be sufficiently regular (at least $C^3$).
Equivalently, $\Phi$ solves the corresponding Monge-Amp{\'e}re  equation (\ref{MA}).

Sometimes we assume that $\mu$ is supported on a convex  domain  $\Omega$.
The domain $\Omega$ is called proper if it does not contain a full line.

The function $\Phi$ can arise as a solution to any of the following problems:
\begin{enumerate}
\item
Optimal transportation problem: Given probability measures $\mu, \nu$
find the (unique up to a set of $\mu$-measure zero) function $\Phi$ solving (\ref{MA}). We refer to \cite{BoKo2012}, \cite{Vill},
where the reader can find comprehensive information about the solvability, uniqueness, and regularity issues.
\item
Given a probability measure $\nu$ satisfying $\int x d \nu=0$ find a solution
to the elliptic K\"ahler-Einstein equation
$$
e^{-\Phi} = e^{-W(\nabla \Phi)} \det D^2 \Phi.
$$
The existence and uniqueness results are  presented in \cite{BB, CoKl, WZ}.
\item
Given a proper open convex domain $\Omega \subset \mathbb{R}^n$ find a solution
to the hyperbolic K\"ahler-Einstein equation
\begin{equation} \label{KE-hyp}
e^{\Phi} = \det D^2 \Phi.
\end{equation}
on $\Omega$ which satisfies  $\lim_{x \to \partial \Omega} \Phi(x) = + \infty$.
\end{enumerate}

The well-posedness of (\ref{KE-hyp}) was proved by Cheng and Yau, who
continued the investigations of Calabi and Nirenberg. The formulation below is taken from \cite{Fox}.
\begin{theorem}(S.Y. Cheng, S.T. Yau, \cite{CY1}, \cite{CY2}, \cite{CY3})
\label{CYth}
For every proper open convex domain $\Omega \subset \mathbb{R}^n$ there exists a unique convex function $\Phi$ solving
(\ref{KE-hyp}) and satisfying $\lim_{x \to \partial \Omega} \Phi(x) = + \infty$.
The Riemannian metric $
h = D^2 \Phi
 $
 is complete on $\Omega$.
\end{theorem}

We assume throughout this paper that we are given the standard Euclidean coordinate system  $\{x^i\}$.
The interior of  $\Omega = \mbox{supp}(\mu)$ is equipped with the metric
$$
h = h_{ij} dx^i dx^j = \Phi_{ij} dx^i dx^j = (\partial^2_{x_i x_j} \Phi) dx^i dx^j
$$
and with the measure $\mu$.
The Legendre transform
$$
\Psi(y) = \sup_{x \in \Omega} \bigl( \langle x, y \rangle - \Phi(x) \bigr)
$$
defines the dual convex potential $\Psi$, satisfying $\nabla \Phi \circ \nabla \Psi(y) =y$ and pushing forward  $\nu$ onto $\mu$.

We give below a list of useful computational formulas, the reader can find the proofs in \cite{kol}.
 It is convenient to use the following notation:
 $$
 V_i = \partial_{x_i} V, \ V_{ij} = \partial^2_{x_i x_j} V, \ V_{ijk} = \partial^3_{x_i x_j x_k} V
 $$
  $$
 W^i = (\partial_{x_i} W) \circ \nabla \Phi, \ W^{ij} = (\partial^2_{x_i x_j} W)\circ \nabla \Phi, \ W^{ijk} = (\partial^3_{x_i x_j x_k} W)\circ \nabla \Phi.
 $$
 We follow the standard conventions in  Riemannian geometry (i.e., $\Phi^{ij}$ is inverse to $\Phi_{ij}$, Einstein summation, raising indices etc.).
 The measure $\mu$ has the following density with respect to the Riemannian volume
 $$
 \mu = e^{-P} d vol_{g}, \ \ P = \frac{1}{2} \bigl( V + W(\nabla \Phi) \bigr).
 $$

 The associated diffusion generator (weighted Laplacian) $L$ has the form
  $$
L f = \Phi^{ij} f_{ij} - W^i f_i = \Delta_h f - \frac{1}{2}(V^i + W^i) f_i,
 $$
 where $\Delta_h$ is the Riemannian Laplacian.
 The following non-negative symmetric tensor $g$  plays prominent role in our analysis
 $$
 g_{ij} = \Phi_{iab} \Phi_j^{ab}.
 $$

  In order to distinguish  the (weighted) Laplacian of a tensor $T$ and the Laplacian
  of its component in the fixed Euclidean coordinate system we
 use for the latter the  square brackets.
 For instance $(L T)_{ij}$ will denote the Laplacian of the $(0,2)$-tensor $T$
 as in (\ref{eq_new}) while  $$L[T_{ij}]$$ denotes the Laplacian of the scalar function $T_{ij}$ with fixed indices $i,j$.
  The proof of the following lemma can be found in \cite{kol}.

  \begin{lemma}
 \label{Lpd}
   The weighted Laplacians of the partial derivatives of $\Phi$  for fixed $i,j,k$ satisfy the following relations:
   \begin{equation}
   \label{d1}
  L [\Phi_i] = - V_i =  -W_i + \Phi^{kl} \Phi_{ikl}.
  \end{equation}
  \begin{equation}
  \label{d2}
  L [\Phi_{ij}] = -V_{ij} + W_{ij} + g_{ij}
  \end{equation}
  \begin{align}
\label{d3}
L [\Phi_{ijk}] = & - V_{ijk} + W_{ijk}
+ \bigl(  W^{s}_{i} \Phi_{sjk} + W^{s}_{j} \Phi_{sik} + W^{s}_{k} \Phi_{sij} \bigr)
\\& \nonumber
+ \bigl(\Phi_{abi} \Phi^{ab}_{jk} + \Phi_{abj} \Phi^{ab}_{ik} +  \Phi_{abk} \Phi^{ab}_{ij} \bigr) - 2 \Phi^{a}_{bi} \Phi^{b}_{cj} \Phi^{c}_{ak}.
\end{align}
 \end{lemma}

\medskip Recall that for  two tensors $T_{ij}, S_{ij}$, their symmetric product is defined as follows:
 $$
 (T \odot S)_{ij} = \frac{1}{2} \bigl( T_{ik} S^{k}_{j} + T_{jk} S^{k}_{i} \bigr).
 $$
Finally, we give a list  of formulas for the most important quantities.

 \begin{enumerate}
\item Connection
 $$
 \Gamma^k_{ij} = \frac{1}{2} \Phi^k_{ij}.
 $$
 \item Hessian of $f$
 $$
 \nabla^2_h f_{ij} = f_{ij} - \frac{1}{2} \Phi^k_{ij} f_k.
 $$
\item  Riemann tensor
 $$
\mbox{\rm{R}}_{ikjl} = \frac{1}{4} (\Phi_{ila} \Phi^{a}_{kj} - \Phi_{ija} \Phi^{a}_{kl}).
$$
\item Ricci tensor
$$
\mbox{Ric}_{ij} = \frac{1}{4}  \Bigl( \Phi_{iab} \Phi^{ab}_j + \Phi_{ijk} (V^{k}-W^k) \Bigr) = \frac{1}{4} (g_{ij} + \Phi_{ijk}(V^{k}-W^k) ).
$$
\item Bakry-Emery tensor
$$
(\mbox{Ric}_{\mu})_{i,j}  = \mbox{Ric}_{ij}  + \frac{1}{2} \nabla^2_h (V + W(\nabla \Phi))_{ij}=  \frac{1}{4} g_{ij} + \frac{1}{2} V_{ij} +  \frac{1}{2} W_{ij}.
$$
\end{enumerate}

  \section{Laplacians for tensors}

  This section is devoted to computations of the weighted Laplacian of several important tensors.
  We stress that in this section {\bf we omit the subscript $h$} for the sake of simplicity, i.e.
  the symbols $\nabla, \nabla^2$ etc. are always related to the Hessian metric $h$, but not to the Euclidean metric.
  Given a tensor $T$ we define its Laplacian as follows:
  $$
  \Delta T = \Phi^{pq} \nabla_p \nabla_q T.
  $$
   Here $\nabla_p T$ is the covariant derivative, which means, in particular, that
   $$
   \nabla_p \Phi_{ij}=0, \ \Delta \Phi_{ij}=0.
   $$
   Similarly
   $$
   LT = \Delta T - \frac{1}{2} (V^k + W^k) \nabla_k T.
   $$
    \begin{lemma}
Let
$$f_i = \partial_{x_i} f$$
for some function $f$. Then
   $$
\nabla_p f_i = f_{ip} - \frac{1}{2} \Phi_{ip}^k f_k
$$
$$
\nabla_q \nabla_p f_i
= \bigl( f_{ip} - \frac{1}{2} \Phi_{ip}^k f_k \bigr)_q - \frac{1}{2} \Phi^m_{qi} \bigl(f_{mp} - \frac{1}{2} \Phi_{mp}^k f_k \bigr) - \frac{1}{2} \Phi^m_{qp} \bigl(f_{mi} - \frac{1}{2} \Phi_{mi}^k f_k\bigr)
$$
Taking the trace we get
$$
\Delta f_i = \Phi^{pq}\bigl( f_{ip} - \frac{1}{2} \Phi_{ip}^k f_k\bigr)_q - \frac{1}{2} \Phi^{mp}_{i} \bigl(f_{mp} - \frac{1}{2} \Phi_{mp}^k f_k\bigr) +  \frac{1}{2}(V^m - W^m) \bigl(f_{mi} - \frac{1}{2} \Phi_{mi}^k f_k\bigr)
$$
Rearranging the terms we finally obtain using Lemma \ref{Lpd}
\begin{align*}
L f_i  &= \Phi^{pq}\bigl(  f_{ipq} - \frac{1}{2} \Phi_{ipq}^k f_k + \frac{1}{2} \Phi^{km}_{q} \Phi_{imp} f_k - \frac{1}{2} \Phi_{ip}^k f_{kq} \bigr)
 - \frac{1}{2} \Phi^{mp}_i f_{mp}+ \frac{1}{4} g_{i}^k f_k - W^m (f_{mi} - \frac{1}{2} \Phi_{mi}^k f_k)
 \\& =  L[f_i] +  W^k f_{ik} - \frac{1}{2}(  L[\Phi_{ik}] + W^p \Phi_{ikp}) f^k - \Phi_{i}^{mk} f_{mk} + \frac{3}{4} g_{i}^k f_k - W^m f_{mi}  +  \frac{1}{2} W^m \Phi_{mi}^k f_k
 \\&
=     L[f_i]  -  \frac{1}{2} L[\Phi_{ik}] f^k  + \frac{3}{4} g_{i}^k f_k - \Phi_{i}^{mk} f_{mk} =    L[f_i]   - \Phi_{i}^{mk} f_{mk} +  \frac{1}{2} (V_{ik} - W_{ik}) f^k +  \frac{1}{4} g_{i}^k f_k  .
\\& = \Phi^{mk} f_{imk}  - \Phi_{i}^{mk} f_{mk}  - W_{k} f^{k}_{i} + \frac{1}{2} (V_{ik} - W_{ik}) f^k +  \frac{1}{4} g_{i}^k f_k
\end{align*}
   \end{lemma}
   \begin{corollary}
$$
L \Phi_i =  \frac{1}{2} (V_{ik} - W_{ik} )\Phi^k + \frac{1}{4} g_{i}^k \Phi_k -  W_{i}.
$$
\end{corollary}
   All of the following calculations are essentially based on the next Lemma, which is obtained by direct computations with the help of Lemma \ref{Lpd}. The computation is long and quite standard.
   See Section 4 for a graphical method for performing this computation relatively quickly.

   \begin{lemma}
      \begin{align*}
L \Phi_{iab} & = - V_{iab} + W_{iab} + \frac{1}{2} \bigl( (V_i^m + W_i^m) \Phi_{mab} + (V_a^m + W_a^m) \Phi_{mib} + (V_b^m + W_b^m) \Phi_{mia} \bigr)
\\& - \frac{1}{2} \Phi^{l}_{ik} \Phi^{m}_{al} \Phi^{k}_{bm}
+ \frac{1}{4} \Bigl( g^{k}_{i} \Phi_{kab} +  g^{k}_{a} \Phi_{kib} +  g^{k}_{b} \Phi_{kia} \Bigr).
\end{align*}
\end{lemma}

\begin{proposition}
\label{lgij}
\begin{align*}
L g_{ij} & =  (-V_{iab} + W_{iab}) \Phi^{ab}_j +  (-V_{jab} + W_{jab}) \Phi^{ab}_i
\\&   + \frac{1}{2} \Bigl(  (V_{is} + W_{is}) g^{s}_j + (V_{js} + W_{js}) g^{s}_i\Bigr)+ 2(V_{am} + W_{am}) \Phi^{m}_{ib} \Phi^{ab}_j
\\&  +\frac{1}{2} g_{ki} g^{k}_{j} + 2 \nabla_p \Phi_{iab} \nabla^p  \Phi_{j}^{ab}
+    8 \mbox{\rm{R}}_{iabc} \mbox{\rm{R}}_{j}^{abc}.
\end{align*}
In particular, if $\Phi$ satisfies (\ref{HKE}) then
$$
L g_{ij} =   \bigl( g_{ki} (\mbox{\rm{Ric}}_{\mu})^{k}_{j} + g_{kj} (\mbox{\rm{Ric}}_{\mu})^{k}_{i}\bigr)+ 2 \nabla_p \Phi_{iab} \nabla^p  \Phi_{j}^{ab} + 8 \mbox{\rm{R}}_{iabc} \mbox{\rm{R}}_{j}^{abc},
$$
which implies
$$
L g \ge 2 g \odot \mbox{\rm{Ric}}_{\mu}.
$$
\end{proposition}
\begin{proof}
Applying $L$  to $g_{ij} = \Phi_{iab} \Phi_{j}^{ab}$, one gets
$$
L(g_{ij}) =( L \Phi_{iab}) \Phi^{ab}_j + 2 \nabla_p \Phi_{iab} \nabla^p \Phi_{j}^{ab} +
\Phi_{iab} ( L \Phi^{ab}_j ).
$$
Lemma \ref{lgij} implies
\begin{align*}
(L \Phi_{iab}) \Phi^{ab}_j
& =
\Bigl[ - V_{iab} + W_{iab} + \frac{1}{2} \bigl( (V_i^m + W_i^m) \Phi_{mab} + (V_a^m + W_a^m) \Phi_{mib} + (V_b^m + W_b^m) \Phi_{mia} \bigr) \Bigr] \Phi^{ab}_j
\\& - \frac{1}{2} \Phi^{l}_{ik} \Phi^{m}_{al} \Phi^{k}_{bm} \Phi^{ab}_j
+ \frac{1}{4} \Bigl( g^{k}_{i} \Phi_{kab} +  g^{k}_{a} \Phi_{kib} +  g^{k}_{b} \Phi_{kia} \Bigr) \Phi^{ab}_j.
\end{align*}
The similar formula for $(L \Phi_{jab}) \Phi^{ab}_i$ is obtained by interchanging $i$ and $j$. Using the relations
$g^k_i \Phi_{kab} \Phi^{ab}_j = g^{k}_{i} g_{kj}$ and
 $g^{k}_{a} \Phi_{kib} \Phi_{j}^{ab} -  \Phi^{l}_{ik} \Phi^{m}_{al} \Phi^{k}_{bm} \Phi^{ab}_{j}
= 8 \mbox{\rm{R}}_{iabc} \mbox{\rm{R}}_{j}^{abc}$ one gets
\begin{align*}
(L \Phi_{iab}) \Phi^{ab}_j
& =
\Bigl[ - V_{iab} + W_{iab} + \frac{1}{2} \bigl( (V_i^m + W_i^m) \Phi_{mab} + (V_a^m + W_a^m) \Phi_{mib} + (V_b^m + W_b^m) \Phi_{mia} \bigr) \Bigr] \Phi^{ab}_j
\\& + \frac{1}{4} g^{k}_{i} g_{kj} + 4 \mbox{\rm{R}}_{iabc} \mbox{\rm{R}}_{j}^{abc}
\end{align*}
and the claim follows.
\end{proof}

 \section{Computations with diagrams}

 Some of the computations of the previous sections are rather tedious, and some of the formulas
 are
 not very pleasant to the eye. Consider, for example, the following expression from formula (\ref{d3}):
 $$ - V_{ijk} + W_{ijk}
+ \bigl(  W^{s}_{i} \Phi_{sjk} + W^{s}_{j} \Phi_{sik} + W^{s}_{k} \Phi_{sij} \bigr)
+ \bigl(\Phi_{abi} \Phi^{ab}_{jk} + \Phi_{abj} \Phi^{ab}_{ik} +  \Phi_{abk} \Phi^{ab}_{ij} \bigr) - 2 \Phi^{a}_{bi} \Phi^{b}_{cj} \Phi^{c}_{ak}. $$
We propose to replace it  by the diagram in Figure 1.

\begin{figure}[h!]
\begin{center}

\begin{tikzpicture}[show background rectangle,>=stealth',shorten >=1pt,auto,node distance=1.6cm,
  thick,main node/.style={circle,draw,inner sep=1pt,minimum size=1pt}]

  \node[main node] (A) {$V$};
  \draw (A) -- ++(-0.5,-1.4);
  \draw (A) -- ++(0,-1.4);
  \draw (A) -- ++(0.5,-1.4);
  \node at ($(A) + (0, -2)$) {$-1$};

  \node at ($(A) + (0.9, -0.65)$) {$+$};

  \node[main node] (B) [right=1.3 of A] {$W$};
  \draw (B) -- ++(-0.5,-1.4);
  \draw (B) -- ++(0,-1.4);
  \draw (B) -- ++(0.5,-1.4);
  \node at ($(B) + (0, -2)$) {$1$};

  \node at ($(B) + (1.0, -0.65)$) {$+$};

  \node[main node] (C) [right=1.3 of B] {$W$};
  \node[main node] (D) [right of=C] {$\Phi$};
  \draw (C) -- (D);
  \draw (C) -- +(0,-1.4);
  \draw (D) -- +(0.3,-1.4);
  \draw (D) -- +(-0.3,-1.4);
  \node at ($(C) + (0.8, -2)$) {$3$};

  \node at ($(C) + (2.5, -0.65)$) {$+$};

  \node[main node] (E) [right=1.3 of D] {$\Phi$};
  \node[main node] (F) [right of=E] {$\Phi$};
  \draw (E) edge [bend right] node[left] {} (F);
  \draw (E) edge [bend left] node[left] {} (F);
  \draw (E) -- ++(0,-1.4);
  \draw (F) -- ++(0.3,-1.4);
  \draw (F) -- ++(-0.3,-1.4);
  \node at ($(E) + (0.8, -2)$) {$3$};

  \node at ($(E) + (2.5, -0.65)$) {$+$};

  \node[main node] (G) [right=1.3 of F] {$\Phi$};
  \node[main node] (H) [below right= 0.3cm and 0.4cm of G] {$\Phi$};
  \node[main node] (I) [right of=G] {$\Phi$};
  \draw (G) -- (H);
  \draw (H) -- (I);
  \draw (I) -- (G);
  \draw (G) -- ++(0,-1.4);
  \draw (H) -- ++(0,-0.7);
  \draw (I) -- ++(0,-1.4);
  \node at ($(G) + (0.8, -2)$) {$-2$};

\end{tikzpicture}
\end{center}
\caption{}
\end{figure}

The diagram in Figure 1 is the weighted sum of five basic diagrams, the number below each basic diagram is its coefficient.
A basic diagram $D$ consists of a set of vertices $V = V(D)$ and two collections of edges $E_{int} = E_{int}(D)$ and $E_{ext} = E_{ext}(D)$.
Each vertex is marked with a letter, which is usually either $\Phi, V$ or $W$. An internal edge $e \in E_{int}$
connects two vertices $x, y \in V$. An external edge is connected only to a single vertex.
To each basic diagram $D$ with $L = \#(E_{ext})$ there
corresponds a symmetric $(0,L)$-tensor
 constructed via the following mechanism:

\begin{enumerate}
\item Associate a new index with any edge.
Assume that the indices associated with the external edges are $i_1,\ldots,i_L$
and that those associated with internal edges are $i_{L+1},\ldots,i_{L + L'}$.
\item Orient all of the internal edges in an arbitrary manner.
For each vertex $v$, let $v^{up}$ (respectively, $v_{down}$)
be the set of all indices
of edges arriving at $v$ (respectively, emanating from $v$). Let $v_{ext}$ be the set of all indices of external edges connected to $v$.
\item For a vertex $v$  write $S(v)$ for the letter with which it is marked.
Write $S_L$ for the collection of all permutations of $\{1,\ldots,L\}$.
For a permutation $\sigma \in S_L$
set $\sigma(v_{ext}) = \{ i_{\sigma(j)} \, ; \, i_j \in v_{ext} \}$.
\item
The resulting symmetric tensor is
\begin{equation} \frac{1}{L!} \sum_{\sigma \in S_L} \prod_{v \in V} S(v)^{v^{up}}_{v_{down}, \sigma(v_{ext})} dx^{i_1} \ldots dx^{i_L}. \label{eq_546} \end{equation}
\end{enumerate}

\medskip Note that If $v$ is a vertex marked by $\Phi$ whose degree is two, then $v$ is the middle point of a certain path, and
we may contract this path to an edge and obtain an equivalent diagram.

\medskip 
We may also accommodate non-symmetric tensors, by marking the external edges of the diagram with the indices
$i_1,\ldots,i_L$. In this case, the tensor corresponding to the diagram is constructed in the same way, except that in (\ref{eq_546}), we always take $\sigma$
to be the identity permutation, i.e., there is no need for a sum and for the normalizing $1/L!$ factor.
From our experience, after a bit of training it is easier to compute
the symmetric contraction product, the covariant derivative and the weighted Laplacian of a tensor in terms of these diagrams.

\medskip We proceed to describe the contraction product of  two basic diagrams, an example is presented in Figure 2.

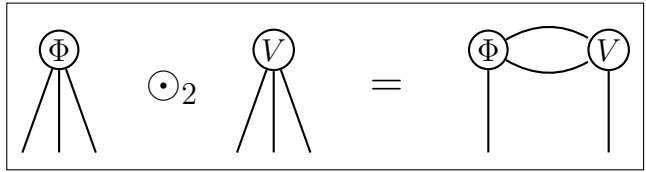
\begin{figure}[h!]
\begin{center}

\begin{tikzpicture}[show background rectangle,>=stealth',shorten >=1pt,auto,node distance=1.6cm,
  thick,main node/.style={circle,draw,inner sep=1pt,minimum size=1pt}]

  \node[main node] (A) {$\Phi$};
  \draw (A) -- ++(-0.5,-1.4);
  \draw (A) -- ++(0,-1.4);
  \draw (A) -- ++(0.5,-1.4);
  \node at ($(A) + (1.5, -0.5)$) {\Large $\odot_2$};

  \node[main node] (B) [right=2.3 of A] {$V$};
  \draw (B) -- ++(-0.5,-1.4);
  \draw (B) -- ++(0,-1.4);
  \draw (B) -- ++(0.5,-1.4);
  \node at ($(B) + (1.5, -0.5)$) {\Large $=$};

  \node[main node] (C) [right=2.3 of B] {$\Phi$};
  \node[main node] (D) [right of=C] {$V$};
  \draw (C) edge [bend right] node[left] {} (D);
  \draw (C) edge [bend left] node[left] {} (D);
  \draw (C) -- ++(0,-1.4);
  \draw (D) -- ++(0,-1.4);

\end{tikzpicture}
\end{center}
\caption{The tensor $\Phi_{ik \ell} V_j^{k \ell}$, which is the symmetric contraction product of $\Phi_{ijk}$ and $V_{ijk}$, where we contract two indices.}
\end{figure}

Formally, assume that we are given two basic diagrams $D_1$ and $D_2$.
Set $L_i = \#(E_{ext}(D_i))$ for $i=1,2$, and assume that $L_1 \geq k$ and $L_2 \geq k$
where $k \geq 1$ is an integer.
The symmetric contraction  product $D_1 \odot_k D_2$
is described as follows:

\begin{enumerate}
\item For each subset $A \subseteq E_{ext}(D_1)$ and $B \subseteq E_{ext}(D_2)$, both of size exactly $k$,
and for each  invertible map $f: A \rightarrow B$ we construct a basic diagram. The weight of this  basic
diagram is $1 / n$, where $n = (L_1! L_2!) / (k! (L_1 - k)! (L_2 - k)! )$ is the number of all possible choices of $A,B$ and $f$.

\item The basic diagram corresponding to $A, B$ and $f$, is the disjoint union of $D_1$ and $D_2$, except that for any $e \in A$, we replace the  pair of edges $e \in A$ and $f(e) \in B$
by a single edge connecting the vertex of $e$ and the vertex of $f(e)$.
\end{enumerate}

The integer $k$ in the notation $\odot_k$ is referred to as the order of the symmetric contraction product.
Recall that the contraction of two indices in a given tensor corresponds to a symmetric contraction product with the tensor $\Phi_{ij}$.

\medskip We move on to the description of covariant differentiation. The covariant derivative of a symmetric tensor is not necessarily symmetric.

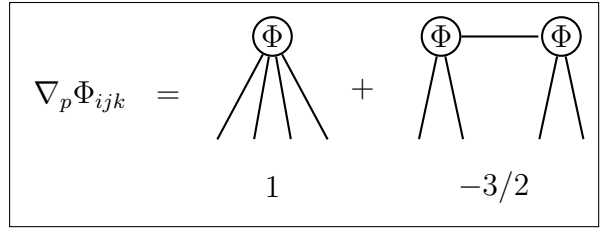
\begin{figure}[h!]
\begin{center}

\begin{tikzpicture}[show background rectangle,>=stealth',shorten >=1pt,auto,node distance=1.6cm,
  thick,main node/.style={circle,draw,inner sep=1pt,minimum size=1pt}]

  \node[main node] (A) {$\Phi$};
  \draw (A) -- ++(0.75,-1.4);
  \draw (A) -- ++(0.25,-1.4);
  \draw (A) -- ++(-0.25,-1.4);
  \draw (A) -- ++(-0.75,-1.4);
  \node at ($(A) + (0, -2)$) {$1$};

  \node at ($(A) + (-2.2, -0.8)$) {$\nabla_p \Phi_{ijk} \ \  =$};

  \node at ($(A) + (1.2, -0.7)$) {$+$};

  \node[main node] (B) [right=1.7 of A] {$\Phi$};
  \node[main node] (C) [right of=B] {$\Phi$};
  \draw (B) -- ++(-0.3,-1.4);
  \draw (B) -- ++(0.3,-1.4);
  \draw (B) -- (C);
  \draw (C) -- ++(-0.3,-1.4);
  \draw (C) -- ++(0.3,-1.4);
  \node at ($(B) + (0.7, -2)$) {$-3/2$};
\end{tikzpicture}
\end{center}
\caption{The covariant derivative $\nabla_p (\Phi_{ijk})$ has turned out to be a symmetric tensor.}
\end{figure}

Here are the general rules for depicting  the covariant derivative $\nabla_p$ of the basic diagram $D$:

\begin{enumerate}
\item For each vertex $v$, we add a basic diagram $D_v$ whose weight is $+1$. The basic diagram $D_v$ is constructed from $D$
by adding an external edge emanating from $v$ and marked by $p$.
\item For each internal edge $e$, we add a basic diagram $D_e$ whose weight is $-1$. The basic diagram $D_e$ is constructed from $D$
by adding an external edge, marked by $p$, which is emanating from a new vertex, marked by $\Phi$, in the middle of the edge $e$.
\item For each external edge $e$ we add a basic diagram whose weight is $-1/2$, which is constructed exactly as in the case of an internal edge.
\end{enumerate}

An internal edge $e \in E_{int}$ is called a loop if it connects a vertex to itself.
The Monge-Amp\`ere equation (\ref{MA}) allows us to eliminate any loop which connects a vertex marked by $\Phi$ to itself.
For example, by differentiating (\ref{MA}) we obtain
$$ \Phi_{ji}^j = -V_i + W_i, \qquad \Phi_{ijk}^k = -V_{ij} + W_{ij} + \Phi_{i k}^{\ell} \Phi_{j \ell}^k + \Phi_{ij}^k W_k. $$
Thus, we may replace a $\Phi$-vertex having  a loop and additional $k$ edges by a certain sum of basic diagrams.
The rules for loop elimination in the case where $k=1,2,3$ are depicted in Figure 4.

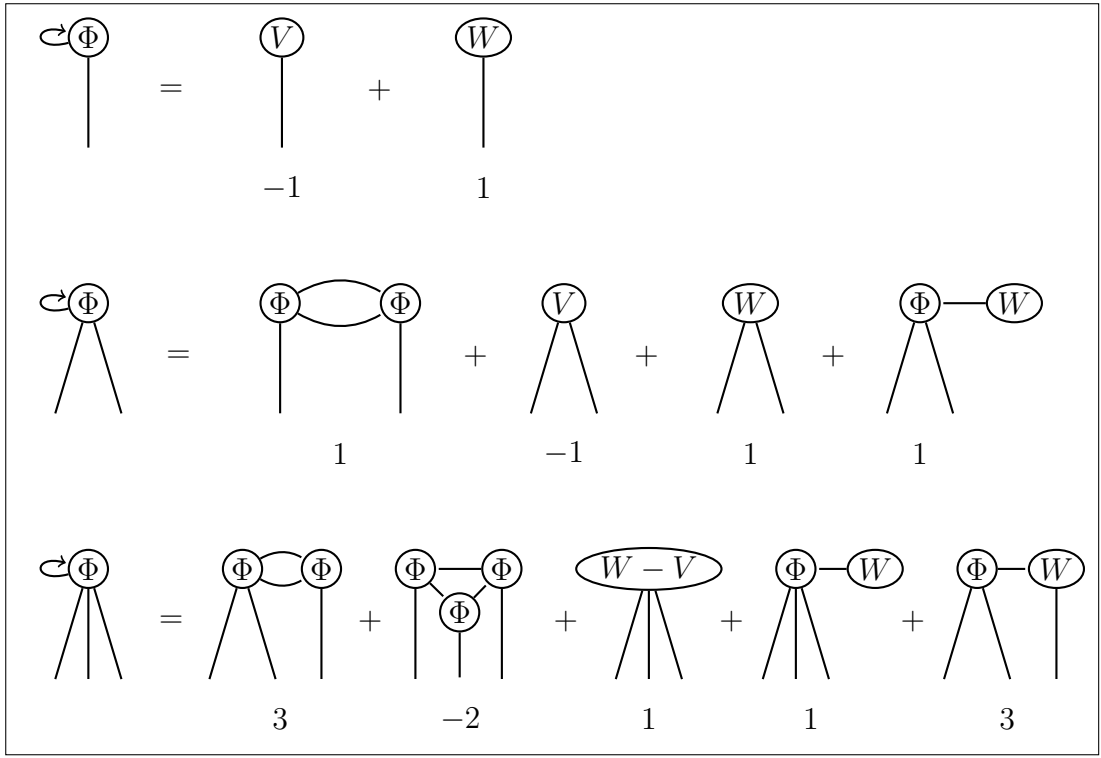
\begin{figure}[h!]
\begin{center}

\begin{tikzpicture}[show background rectangle,shorten >=1pt,auto,node distance=1.6cm,
  thick,main node/.style={ellipse,draw,inner sep=1pt,minimum size=1pt}]

  \node[main node] (A) {$\Phi$};
  \draw[] (A) -- ++(0, -1.5);
  \path[every node/.style={font=\sffamily\small}]
    (A) edge [loop left] node {} (A);

  \node at ($(A) + (1.1, -0.7)$) {$=$};

  \node[main node] (B) [right = 2.0 of A] {$V$};
  \draw[] (B) -- ++(0,-1.5);
  \node at ($(B) + (0, -2)$) {$-1$};

  \node at ($(B) + (1.3, -0.7)$) {$+$};

  \node[main node] (C) [right = 2 of B] {$W$};
  \draw[] (C) -- ++(0,-1.5);
  \node at ($(C) + (0, -2)$) {$1$};

  \node[main node] (D) [below = 3 of A] {$\Phi$};
  \draw [](D) -- ++(0.45, -1.5);
  \draw [](D) -- ++(-0.45, -1.5);
  \path[every node/.style={font=\sffamily\small}]
    (D) edge [loop left] node {} (D);

  \node at ($(D) + (1.2, -0.7)$) {$=$};

  \node[main node] (E) [right=2 of D] {$\Phi$};
  \node[main node] (F) [right of=E] {$\Phi$};
  \draw (E) edge [bend right] node[left] {} (F);
  \draw (E) edge [bend left] node[left] {} (F);
  \draw[] (E) -- ++(0, -1.5);
  \draw[] (F) -- ++(0, -1.5);
  \node at ($(E) + (0.8, -2)$) {$1$};

  \node at ($(F) + (1, -0.7)$) {$+$};

  \node[main node] (G) [right = 1.6 of  F] {$V$};
  \draw [](G) -- ++(0.45, -1.5);
  \draw [](G) -- ++(-0.45, -1.5);
  \node at ($(G) + (0, -2)$) {$-1$};

  \node at ($(G) + (1.1, -0.7)$) {$+$};

  \node[main node] (H) [right = 1.8 of  G] {$W$};
  \draw [](H) -- ++(0.45, -1.5);
  \draw [](H) -- ++(-0.45, -1.5);
  \node at ($(H) + (0, -2)$) {$1$};

  \node at ($(H) + (1.1, -0.7)$) {$+$};

  \node[main node] (J) [right = 1.6 of  H] {$\Phi$};
  \node[main node] (I) [right = 0.6 of  J] {$W$};
  \draw (I) -- (J);
  \draw [](J) -- ++(0.45, -1.5);
  \draw [](J) -- ++(-0.45, -1.5);
  \node at ($(J) + (-0, -2)$) {$1$};

  \node[main node] (K) [below = 3 of D] {$\Phi$};
  \draw [](K) -- ++(0.45, -1.5);
  \draw [](K) -- ++(-0.45, -1.5);
  \draw [](K) -- ++(0, -1.5);
  \path[every node/.style={font=\sffamily\small}]
    (K) edge [loop left] node {} (K);

  \node at ($(K) + (1.1, -0.7)$) {$=$};

  \node[main node] (L) [right=1.5 of K] {$\Phi$};
  \node[main node] (M) [right=0.5 of L] {$\Phi$};
  \draw (L) edge [bend right] node[left] {} (M);
  \draw (L) edge [bend left] node[left] {} (M);
  \draw[] (L) -- ++(0.45, -1.5);
  \draw[] (L) -- ++(-0.45, -1.5);
  \draw[] (M) -- ++(0, -1.5);
  \node at ($(L) + (0.5, -2)$) {$3$};

  \node at ($(L) + (1.7, -0.7)$) {$+$};

  \node[main node] (N) [right=0.7 of M] {$\Phi$};
  \node[main node] (O) [below right= 0.2cm and 0.2cm of N] {$\Phi$};
  \node[main node] (P) [right= 0.6 of N] {$\Phi$};
  \draw (N) -- (O);
  \draw (O) -- (P);
  \draw (P) -- (N);
  \draw [] (N) -- ++(0,-1.5);
  \draw [] (O) -- ++(0,-0.9);
  \draw [] (P) -- ++(0,-1.5);
  \node at ($(N) + (0.6, -2)$) {$-2$};

  \node at ($(N) + (2, -0.7)$) {$+$};

  \node[main node] (Q) [right=0.7 of P] {$W-V$};
  \draw [](Q) -- ++(0.45, -1.5);
  \draw [](Q) -- ++(-0.45, -1.5);
  \draw [](Q) -- ++(0, -1.5);
  \node at ($(Q) + (0, -2)$) {$1$};

  \node at ($(Q) + (1.1, -0.7)$) {$+$};

  \node[main node] (R) [right = 0.7 of  Q] {$\Phi$};
  \node[main node] (S) [right = 0.4 of  R] {$W$};
  \draw (S) -- (R);
  \draw [](R) -- ++(0.45, -1.5);
  \draw [](R) -- ++(-0.45, -1.5);
  \draw [](R) -- ++(-0, -1.5);
  \node at ($(R) + (0.2, -2)$) {$1$};

  \node at ($(S) + (0.5, -0.7)$) {$+$};

  \node[main node] (T) [right = 0.7 of  S] {$\Phi$};
  \node[main node] (U) [right = 0.4 of  T] {$W$};
  \draw (T) -- (U);
  \draw [](T) -- ++(0.45, -1.5);
  \draw [](T) -- ++(-0.45, -1.5);
  \draw [](U) -- ++(-0, -1.5);
  \node at ($(T) + (0.4, -2)$) {$3$};

\end{tikzpicture}
\end{center}
\caption{Loop elimination \label{fig4}}
\end{figure}

Let us emphasize that when applying the loop elimination rules, we count all the edges related to the loop.
For example, if a $\Phi$-vertex has a loop, plus one external and one internal edge, then the second picture in Figure \ref{fig4} is applicable.
The combinatorics of loop elimination for an arbitrary $k \geq 4$ is not too complicated, but it
will not be needed here.  We move on to the weighted Laplacian $L$.
Here are the  rules for depicting  the weighted Laplacian $L$ of a basic diagram $D$:

\begin{enumerate}
\item For $a \in V \cup E_{ext} \cup E_{int}$ we denote
$$ w(a) = \left \{ \begin{array}{ll} 1 & a \in V \\ -1 & a \in E_{int} \\ -1/2 & a \in E_{ext} \end{array} \right. $$
\item For any $a,b \in V \cup E_{ext} \cup E_{int}$ with $a \neq b$ we add a basic diagram $D_{a,b}$ whose weight is $w(a) w(b)$.
The basic diagram $D_{a,b}$ is constructed from $D$
by adding an internal edge connecting $a$ and $b$. Note that when $a$ is an edge, adding an internal edge introduces a new vertex, marked by $\Phi$, in the middle of the edge $a$.
\item For any $a \in V \cup E_{ext} \cup E_{int}$ we add a basic diagram whose weight is $w(a)$ which
is obtained from $D$ by adding a loop around $a$. Note that when $a$ is an edge, we add a vertex marked by $\Phi$ in the middle of the edge $a$ and a loop around this vertex.
\item For any $a \in V \cup E_{ext} \cup E_{int}$ we add a basic diagram whose weight is $-w(a)$ which is obtained from $D$ by adding a new vertex marked by $W$ and an edge emanating from $a$ to the new vertex.  Note that when $a$ is an edge, the new edge is emanating from a new vertex in the middle of the edge $a$.
\item For any edge $e \in E_{ext} \cup E_{int}$ we add a basic diagram that is obtained from $D$ by adding two vertices on the edge $e$
and connecting one to the other via a new internal edge. The weight of this basic diagram is $2$ if $a \in E_{int}$ and is $3/4$ if $a \in E_{ext}$.
\end{enumerate}

\begin{figure}[h!]
\begin{center}

\begin{tikzpicture}[show background rectangle,>=stealth',shorten >=1pt,auto,node distance=1.4cm,
  thick,main node/.style={circle,draw,inner sep=0pt,minimum size=0pt}]

  \node[main node] (A1) {};
  \node[main node] (A2) [below of=A1] {};
  \node[main node] (A3) [right=0.5 of A1] {};
  \node[main node] (A4) [below of=A3] {};
  \draw (A1) edge [bend right] node[left] {} (A2);
  \draw (A3) edge [bend left] node[left] {} (A4);

  \node[main node] (A5) [right=0.25 of A1] {};
  \node[main node] (A6) [below of= A5] {};
  \draw [fill] (A5) circle [radius=0.05];
  \draw (A5) -- (A6);
  \node at ($(A1) + (-0.5, -0.7)$) {$L$};
  \node at ($(A3) + (0.7, -0.7)$) {$=$};

  \node[main node] (B1) [right of=A3] {};
  \node[main node] (B2) [below of=B1] {};
  \node[main node] (B3) [right= 1.9 of B1] {};
  \node[main node] (B4) [below of=B3] {};
  \draw (B1) edge [bend right] node[left] {} (B2);
  \draw (B3) edge [bend left] node[left] {} (B4);

  \node[main node] (B5) [right=0.25 of B1] {};
  \node[main node] (B6) [below of= B5] {};
  \node[main node] (B7) [above=0.3 of B5] {};
  \draw [fill] (B5) circle [radius=0.05];
  \draw (B5) -- (B6);
  \draw (B5) arc [radius=0.1, start angle=270, end angle= 630];
  \node at ($(B5) + (0.05, 0.45)$) {{\tiny rule 3}};
  \node at ($(B5) + (0.4, -0.7)$) {$+$};
  \node at ($(B6) + (0, -0.4)$) {$1$};

  \node[main node] (B8) [right=1.05 of B1] {};
  \node[main node] (B9) [below of= B8] {};
  \node[main node] (B10) [right=0.3 of B8] {{\tiny W}};
  \draw [fill] (B8) circle [radius=0.05];
  \draw (B8) -- (B9);
  \draw (B8) -- (B10);
  \node at ($(B8) + (0.3, 0.45)$) {{\tiny rule 4}};
  \node at ($(B9) + (0.23, -0.43)$) {$-1$};

  \node at ($(B3) + (0.7, -0.7)$) {$+$};

  \node[main node] (C1) [right of=B3] {};
  \node[main node] (C2) [below of=C1] {};
  \node[main node] (C3) [right= 1.9 of C1] {};
  \node[main node] (C4) [below of=C3] {};
  \draw (C1) edge [bend right] node[left] {} (C2);
  \draw (C3) edge [bend left] node[left] {} (C4);

  \node[main node] (C5) [right=0.25 of C1] {};
  \node[main node] (C6) [below of= C5] {};
  \draw [fill] (C5) circle [radius=0.05];
  \draw [fill] ($(C5) + (0, -0.6)$) circle [radius=0.05];
  \draw (C5) -- (C6);
  \draw ($(C5) + (0, -0.6)$)  arc [radius=0.1, start angle=0, end angle= 360];
  \node at ($(C5) + (0.05, 0.45)$) {{\tiny rule 3}};
  \node at ($(C5) + (0.4, -0.7)$) {$+$};
  \node at ($(C6) + (0, -0.4)$) {$-\frac{1}{2}$};

  \node[main node] (C8) [right=1.05 of C1] {};
  \node[main node] (C9) [below of= C8] {};
  \node[main node] (C10) [below right= 0.5cm and 0.3cm of C8] {{\tiny W}};
  \node[main node] (C11) [below=0.6 of C8] {};
  \draw [fill] (C8) circle [radius=0.05];
  \draw [fill] (C11) circle [radius=0.05];
  \draw (C8) -- (C9);
  \draw (C10) -- (C11);
  \node at ($(C8) + (0.3, 0.45)$) {{\tiny rule 4}};
  \node at ($(C9) + (0.3, -0.4)$) {$\frac{1}{2}$};

  \node at ($(C3) + (0.7, -0.7)$) {$+$};

  \node[main node] (D1) [right of=C3] {};
  \node[main node] (D2) [below of=D1] {};
  \node[main node] (D3) [right=0.7 of D1] {};
  \node[main node] (D4) [below of=D3] {};
  \draw (D1) edge [bend right] node[left] {} (D2);
  \draw (D3) edge [bend left] node[left] {} (D4);

  \node[main node] (D5) [right=0.25 of D1] {};
  \node[main node] (D6) [below= 0.6 of D5] {};
  \node[main node] (D7) [below of= D5] {};
  \draw [fill] (D5) circle [radius=0.05];
  \draw [fill] (D6) circle [radius=0.05];
  \draw (D5) edge [bend left] node[left] {} (D6);
  \draw (D5) edge [bend right] node[left] {} (D6);
  \draw (D6) -- (D7);
  \node at ($(D5) + (0.05, 0.45)$) {{\tiny rule 2}};
  \node at ($(D7) + (0, -0.4)$) {$2 \cdot \left(-\frac{1}{2} \right)$};

  \node at ($(D3) + (0.7, -0.7)$) {$+$};

  \node[main node] (E1) [right of=D3] {};
  \node[main node] (E2) [below= 0.3 of E1] {};
  \node[main node] (E3) [below= 0.9 of E1] {};
  \node[main node] (E4) [below of= E1] {};
  \draw [fill] (E1) circle [radius=0.05];
  \draw [fill] (E2) circle [radius=0.05];
  \draw [fill] (E3) circle [radius=0.05];
  \draw (E1) -- (E2);
  \draw (E2) edge [bend left] node[left] {} (E3);
  \draw (E2) edge [bend right] node[left] {} (E3);
  \draw (E3) -- (E4);
  \node at ($(E1) + (0.05, 0.45)$) {{\tiny rule 5}};
  \node at ($(E4) + (0, -0.4)$) {$\frac{3}{4}$};

  \node at ($(A3) + (0.7, -3.5)$) {$=$};

  \node at ($(A3) + (0.7, -3)$) {\tiny (loop elimination)};

  \node[main node] (F1) [below right= 2.85cm and 1.0cm of B1] {};
  \node[main node] (F2) [below of=F1] {};
  \node[main node] (F3) [right=0.8 of F1] {};
  \node[main node] (F4) [below of=F3] {};
  \draw (F1) edge [bend right] node[left] {} (F2);
  \draw (F3) edge [bend left] node[left] {} (F4);

  \node[main node] (F5) [right=0.25 of F1] {\tiny V};
  \node[main node] (F6) [below of= F5] {};
  \draw (F5) -- (F6);
  \node at ($(F6) + (0, -0.4)$) {$-1$};

  \node at ($(F3) + (0.7, -0.7)$) {$+$};

  \node[main node] (G1) [right of=F3] {};
  \node[main node] (G2) [below of=G1] {};
  \node[main node] (G3) [right= 3.2 of G1] {};
  \node[main node] (G4) [below of=G3] {};
  \draw (G1) edge [bend right] node[left] {} (G2);
  \draw (G3) edge [bend left] node[left] {} (G4);

  \node[main node] (G5) [right=0.25 of G1] {};
  \node[main node] (G6) [below of= G5] {};
  \node[main node] (G7) [right=0.4 of G5] {};
  \node[main node] (G8) [right=0.4 of G7] {};
  \draw [fill] (G5) circle [radius=0.05];
  \draw [fill] (G7) circle [radius=0.05];
  \draw [fill] (G8) circle [radius=0.05];
  \draw (G5) -- (G6);
  \draw (G7) -- (G8);
  \draw (G5) edge [bend right] node[left] {} (G7);
  \draw (G5) edge [bend left] node[left] {} (G7);
  \node at ($(G6) + (0.2, -0.4)$) {$-\frac{1}{2}$};
  \node at ($(G5) + (1.2, -0.7)$) {$+$};

  \node[main node] (G9) [right of= G7] {\tiny V-W};
  \node[main node] (G10) [right=0.5 of G9] {};
  \node[main node] (G11) [below of=G9] {};
  \draw [fill] (G10) circle [radius=0.05];
  \draw (G9) -- (G10);
  \draw (G9) -- (G11);
  \node at ($(G11) + (0.5, -0.4)$) {$\frac{1}{2}$};

  \node at ($(G3) + (0.7, -0.7)$) {$+$};

  \node[main node] (H1) [right of=G3] {};
  \node[main node] (H2) [below of=H1] {};
  \node[main node] (H3) [right=1.2 of H1] {};
  \node[main node] (H4) [below of=H3] {};
  \draw (H1) edge [bend right] node[left] {} (H2);
  \draw (H3) edge [bend left] node[left] {} (H4);

  \node[main node] (H5) [right=0.25 of H1] {\tiny V-W};
  \node[main node] (H6) [below of= H5] {};
  \draw (H5) -- (H6);
   \node at ($(H6) + (0, -0.4)$) {$1$};

  \node at ($(H3) + (0.7, -0.7)$) {$+$};

  \node[main node] (I5) [right of= H3] {};
  \node[main node] (I6) [below of= I5] {};
  \node[main node] (I7) [right=0.4 of I5] {};
  \node[main node] (I8) [right=0.4 of I7] {};
  \draw [fill] (I5) circle [radius=0.05];
  \draw [fill] (I7) circle [radius=0.05];
  \draw [fill] (I8) circle [radius=0.05];
  \draw (I5) -- (I6);
  \draw (I7) -- (I8);
  \draw (I5) edge [bend right] node[left] {} (I7);
  \draw (I5) edge [bend left] node[left] {} (I7);
  \node at ($(I6) + (0.4, -0.4)$) {$\frac{3}{4}$};

\end{tikzpicture}
\caption{Proof of Corollary 3.2 \label{fig112}}
\end{center}
\end{figure}
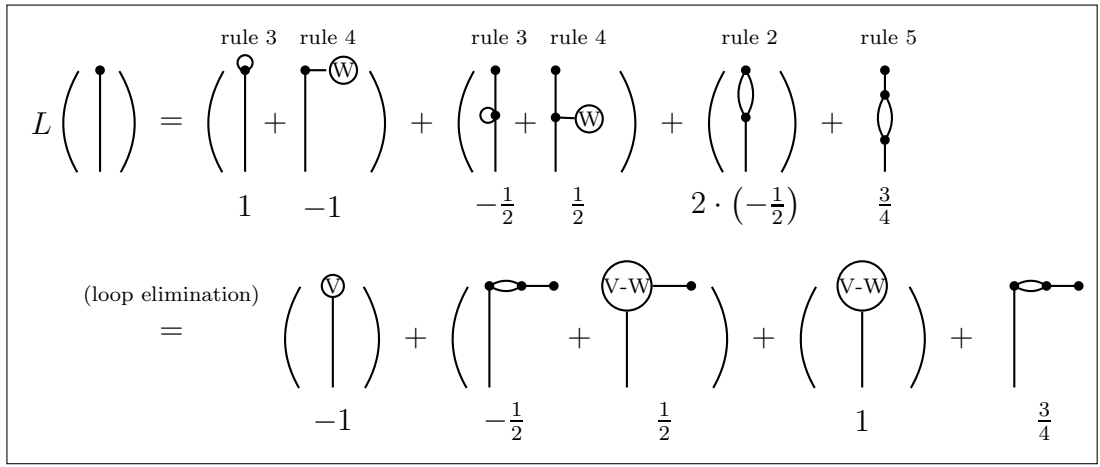

Note that the basic diagram $D_{a,b}$ in the first rule is the same as $D_{b,a}$, hence it appears twice in the resulting diagram.
Figure \ref{fig112} provides an alternative proof of Corollary 3.2 by using the graphical notation.
In order to ease the notation, let us agree that unless specified otherwise, all vertices are marked with $\Phi$ by default.
This makes the diagrams look a bit more clean. The computation of $L \Phi_{ijk}$ in Lemma 3.3 may be depicted in the same way.
In order to verify Proposition 3.4, one needs to compute
the symmetric contraction product of order $2$ of the two tensors $L \Phi_{ijk}$ and $\Phi_{ijk}$.
Figure \ref{fig1447} presents the main expression in Proposition 3.4.

 \begin{figure}[h!]
\begin{center}

\begin{tikzpicture}[show background rectangle,>=stealth',shorten >=1pt,auto,node distance=1.3cm,
  thick,main node/.style={circle,draw,inner sep=1pt,minimum size=1pt}]

  \node[main node] (A) {$V-W$};
  \node[main node] (B) [right of=A] {$\Phi$};
  \draw (A) edge [bend right] node[left] {} (B);
  \draw (A) edge [bend left] node[left] {} (B);
  \draw (A) -- ++(0,-1.4);
  \draw (B) -- ++(0,-1.4);
  \node at ($(A) + (0.6, -2)$) {$-2$};

  \node at ($(A) + (1.9, -0.7)$) {$+$};

  \node[main node] (E) [right= 0.7 of B] {$\Phi$};
  \node[main node] (F) [right= 0.5 of E] {$\Phi$};
  \node[main node] (B) [right= 0.5 of F] {$V+W$};
  \draw (E) edge [bend right] node[left] {} (F);
  \draw (E) edge [bend left] node[left] {} (F);
  \draw (F) -- (B);
  \draw (E) -- ++(0,-1.4);
  \draw (B) -- ++(0,-1.4);
  \node at ($(F) + (0.4, -2)$) {$1$};

  \node at ($(E) + (3.5, -0.7)$) {$+$};

  \node[main node] (G) [right=0.7 of B] {$\Phi$};
  \node[main node] (H) [above right= 0.3cm and 0.3cm of G] {$V+W$};
  \node[main node] (I) [right= 1.5 of G] {$\Phi$};
  \draw (G) -- (H);
  \draw (H) -- (I);
  \draw (I) -- (G);
  \draw (G) -- ++(0,-1.4);
  \draw (H) -- ++(0,-0.7);
  \draw (I) -- ++(0,-1.4);
  \node at ($(G) + (1, -2)$) {$2$};

  \node at ($(I) + (0.55, -0.7)$) {$+$};

  \node[main node] (P) [right= 0.7 of I] {$\Phi$};
  \node[main node] (Q) [right= 0.5 of P] {$\Phi$};
  \node[main node] (R) [right= 0.5 cm of Q] {$\Phi$};
  \node[main node] (S) [right= 0.5 of R] {$\Phi$};
  \draw (P) edge [bend right] node[left] {} (Q);
  \draw (P) edge [bend left] node[left] {} (Q);
  \draw (Q) -- (R);
  \draw (R) edge [bend right] node[left] {} (S);
  \draw (R) edge [bend left] node[left] {} (S);
  \draw (P) -- ++(0,-1.4);
  \draw (S) -- ++(0,-1.4);
  \node at ($(Q) + (0.4, -2)$) {$1/2$};
\end{tikzpicture}

\end{center}
\caption{The tensor $Lg_{ij} - 2 \nabla_p \Phi_{iab} \nabla^p \Phi_j^{ab} - 8 R_{iabc} R_j^{abc}$ \label{fig1447}}
\end{figure}
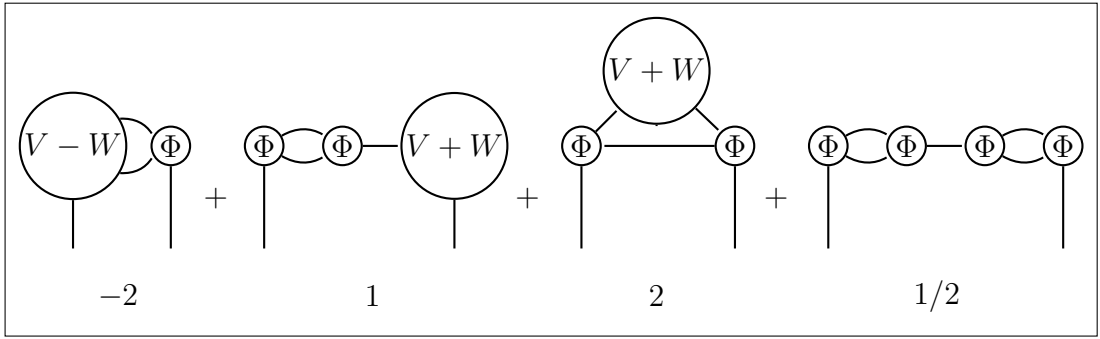

\section{Negativity of $\mbox{Ric}_{\mu}$}

We are ready to prove Theorem  \ref{mainth}.
The existence of the unique solution to (\ref{HKE}) is established in Theorem \ref{CYth}.
We will also apply the following lemma which is a generalization of some classical facts  proved already in the works of Calabi \cite{Calabi1}
and Osserman \cite{Oss}. Extension to the metric-measure space is based on the estimates
for weighted Laplacians of the distance function (see, for instance, \cite{BQ}). 
A more general statement with a proof can be found in \cite{Fox},  Theorem 3.3.

\begin{lemma}
\label{bounded}
Let $(M,g)$ be a complete Riemannian manifold equipped with a measure $\mu = e^{-V} d vol_g$ with twice continuously differentiable density.
Assume that its generalized Bakry-{\'E}mery tensor $\mbox{\rm{Ric}}_{\mu, N} = \mbox{\rm{Ric}}_{\mu} - \frac{1}{N-n} \nabla V \otimes \nabla V$ satisfies
$$
\mbox{\rm{Ric}}_{\mu, N} \ge K
$$
   for some $K \in \mathbb{R}$ and $N > n$.
Let $u \in C^2(M)$ be a non-negative function satisfying
\begin{equation}
Lu \ge Bu^2 - Au \label{eq_new2}
\end{equation}
for some $A,B>0$ at every point $x$ with  $u(x) \ne 0$, where $L = \Delta - \nabla V \cdot \nabla$
is the weighted Laplacian. Then
$$
\sup_{M} u \le \frac{A}{B}.
$$
\end{lemma}

We remark that it is well-known that the assumption that $u$ is a $C^2$-function can be relaxed, and that in 
the case where $u$ is merely continuous, inequality (\ref{eq_new2}) should be interpreted in the viscosity  sense (see, for instance, Definition 2 in \cite{Calabi1}).

\medskip
{\bf Proof of Theorem \ref{mainth}.}
According to Proposition \ref{lgij}
$$
L g \ge 2 g \odot \mbox{\rm{Ric}}_{\mu}.
$$
By taking into account the fact that
$g = 4 \mbox{\rm{Ric}}_{\mu} - {2\alpha} h$, we can rewrite it as follows:
$$
 L  \mbox{\rm{Ric}}_{\mu} \ge 2  \mbox{\rm{Ric}}_{\mu}^2 - {\alpha}  \mbox{\rm{Ric}}_{\mu}.
$$

Let us estimate $L \lambda$ an for arbitrary point $x_0$, where $\lambda$ is the largest eigenvalue of
 $\mbox{\rm{Ric}}_{\mu}$.
One has $\lambda(x_0) =  (\mbox{\rm{Ric}}_{\mu})_{ij}(x_0) \eta^i \eta^j$
for some tangent unit vector $\eta$. Extend $\eta$ in such a way that
$$
\nabla \eta=0, \Delta \eta=0
$$
at $x_0$ (see, for instance,  Theorem 4.6 of \cite{CK}).
Next we note that the function
$
\lambda - (\mbox{\rm{Ric}}_{\mu})_{ij} \eta^i \eta^j
$
is non-negative and equals  zero at $x_0$. Hence $L \lambda -   L ( (\mbox{\rm{Ric}}_{\mu})_{ij} \eta^i \eta^j) \ge 0$
at $x_0$.
One obtains the following relation at $x_0$
$$
L \lambda \ge L ( (\mbox{\rm{Ric}}_{\mu})_{ij} \eta^i \eta^j) = L (\mbox{\rm{Ric}}_{\mu})  \eta^i \eta^j
\ge 2 \lambda^2 - \alpha \lambda.
$$
We will apply Lemma \ref{bounded}.
Note that the completeness of the space follows from Theorem \ref{CYth}.
In addition, it is easy to check (see \cite{kol}) that the tensor $\mbox{\rm{Ric}}_{\mu, 2n}$ is nonnegative,
thus Lemma \ref{bounded} is applicable.
We note that $\tilde{\lambda}= \lambda - \frac{1}{2} \alpha$ is the largest eigenvalue of $\frac{1}{4} g$, hence nonnegative.
The above inequality implies
$$
L \tilde{\lambda} \ge 2  \tilde{\lambda}^2 +  \alpha \tilde{\lambda}.
$$
From Lemma \ref{bounded} it follows that $\tilde{\lambda} \le - \frac{\alpha}{2}$, hence $\lambda \le 0$.




\subsection{Cone case}

Let us analyse the case when $\Omega$ is a cone. Then $\Phi$
is logarithmically homogeneous. This implies, in particular, the following relation:
$$
\Phi_i x^i = 2(n / \alpha)
$$
Differentiating this relation twice one gets
$$
\Phi_{ij} x^j + \Phi_i = 0
$$
$$
2 \Phi_{ij} + \Phi_{ijk} x^k = 0.
$$
We obtain from the first identity $x^j = -\Phi^{ij} \Phi_i$. Substituting this into the second identity, one gets the following Lemma.
\begin{lemma}
\label{zerohess}
Assume that  $\Phi$ is logarithmically homogeneus (this is the case when $\Phi$ solves (\ref{HKE}) in a cone for some negative $\alpha$). Then the following relation holds:
$$
\nabla^2_h \Phi_{ij} = \Phi_{ij} - \frac{1}{2} \Phi^k \Phi_{ijk} =0.
$$
This implies
$$
\nabla^2_h P=\frac{1}{2} \nabla^2_h \Phi_{ij} = 0
$$
and
$$
\mbox{\rm Ric} = \mbox{\rm Ric}_{\mu}
$$
provided $W$ is constant.
\end{lemma}

{\bf Proof of Corollary \ref{sph}}.
The statement $\mbox{\rm Ric} \le 0$ follows immediately from Theorem \ref{mainth} and Lemma \ref{zerohess}.
To prove the statement for a level set $M = \{\Phi = c\}$ we use the following formula (see, for instance, Lemma 7.1
in \cite{KolMil})
$$
\mbox{\rm{Ric}}_{M} = \bigl( \mbox{\rm{Ric}} - \mbox{\rm{R}}(\cdot, \eta, \cdot,\eta) \bigr)|_{TM} + \bigl(H h|_{TM} - II_{M})II_{M}.
$$
 Here $\eta$ is the unit normal to $M$, $II_{M}$ is the second fundamental form of $M$, and $H$ is mean curvature of $M$.
 Since $\nabla^2_h \Phi = 0$ and $M = \{ \Phi = c\}$, necessarily $II_{M}=0$. Note that $\eta = x/|x|_{h}$ and
 $$4\mbox{\rm{R}}(i, x, j,x)=\Phi_{iax} \Phi_{jx}^{a} - \Phi_{axx} \Phi^{a}_{ij}.$$
 Taking into account that $\Phi_{iax}= - 2 \Phi_{ia}$ we easily get $\mbox{\rm{R}}(i, x, j,x) = \nabla^2_h \Phi_{ij}=0$.
Hence $\mbox{\rm{Ric}}_{M} = \mbox{\rm{Ric}}|_{TM} \le 0$.

\section{Application to log-concave measures}
In this section we deal with a couple of probability measures
$$
\mu = e^{-V} dx, \ \nu = e^{-W} dx
$$
and the solution $T = \nabla \Phi$ of the corresponding optimal transportation problem.
The functions $V,W$, and $\Phi$, are assumed to be sufficciently smooth (by the regularity theory for the Monge-Amp\`ere equation the smoothness of $\Phi$ follows from the smoothness of
the potentials under certain convexity assumptions on supports of $\mu, \nu$).

The contraction theorem of L.~Caffarelli has numerous applications in probability and analysis.
It can be stated in the following form (see \cite{Kol2010}, \cite{Kol2011}): under the assumption
$$
D^2 V \le C, \ D^2 W \ge c,
$$
where $c, C$ are positive constants, the potential $\Phi$ satisfies
$$D^2 \Phi \le \sqrt{\frac{C}{c}}.
$$
In this section we prove a kind of extension of this result to the third-order derivatives.

Given a quadratic form $Q$ we denote by $\| Q\|_{h}$ its Riemannian norm:
$$
\| Q\|_{h} = \sup_{v: h(v,v)=1} Q(v,v).
$$

Let us recall that
$$
V_{ij} = \partial^2_{x_i x_j} V, \ V_{ijk} = \partial^3_{x_i x_j x_k} V
$$
{\bf but}
$$
W_{ij} = \sum_{a,b} \Phi_{ai} \Phi_{bj} \cdot \partial^2_{x_b x_a} W, \ W_{ijk} =  \sum_{a,b,c} \Phi_{ai} \Phi_{bj} \Phi_{ck} \cdot \partial^3_{x_a x_b x_c} W.
$$
\begin{proposition}
\label{3-est}
Assume that the matrix with  entries
$V_{ij} + W_{ij}$ is positive semi-definite (this holds, in particular, when both measures are log-concave). Then
the following inequality holds:
\begin{equation}
\label{HLg}
  \| H \|_{h} + 2 L \|g\|_{h} \ge   \|g\|^2_{h}.
\end{equation}
Here
$$
H_{ij} = \mbox{\rm{Tr}} \Bigl[ \bigl[  (V+W)^{(2)} \bigr]^{-1} (V-W)^{(3)}_{i} (D^2 \Phi)^{-1}(V-W)^{(3)}_{j} \Bigr],
$$
$(V-W)^{(3)}_{i}$, $(V+W)^{(2)}$ are the matrices with the entries
$
V_{iab} - W_{iab}
$, $V_{ab} + W_{ab}$ respectively.
\end{proposition}
\begin{proof}
With some abuse of notation, when we write $V_{ij} \geq 0$ we mean that the symmetric matrix $(V_{ij})_{i,j=1}^n$ is positive semi-definite.
According to Proposition \ref{lgij}
\begin{align*}
L g_{ij} & \ge (-V_{iab} + W_{iab}) \Phi^{ab}_j +  (-V_{jab} + W_{jab}) \Phi^{ab}_i \\&
  + \frac{1}{2} \Bigl(  (V_{is} + W_{is}) g^{s}_j + (V_{js} + W_{js}) g^{s}_i\Bigr)+ 2(V_{am} + W_{am}) \Phi^{m}_{ib} \Phi^{ab}_j
 +\frac{1}{2} g_{ki} g^{k}_{j} .
\end{align*}
First we note that
\begin{equation}
\label{VW1}
2 (V_{iab} - W_{iab}) \Phi_i^{ab}
= 2 \mbox{\rm{Tr}} \Bigl[ (D^2 \Phi)^{-1} (V-W)^{(3)}_{i} (D^2 \Phi)^{-1} D^2 \Phi_{e_i}\Bigr],
\end{equation}
\begin{equation}
\label{VW2}
2(V_{am} + W_{am}) \Phi^{m}_{ib} \Phi^{ab}_i =
2 \mbox{\rm{Tr}} \Bigl[ (D^2 \Phi)^{-1} (V+W)^{(2)} (D^2 \Phi)^{-1} A_i \Bigr],
\end{equation}
where   $A_i$ is the  matrix with the entries $\Phi_{iac} \Phi_{ib}^{c}$.
We apply the Cauchy inequality
$$
4\mbox{\rm{Tr}} \bigl( XY \bigr) \le  4 \mbox{\rm{Tr}} \bigl( Q X^2 \bigr) +   \mbox{\rm{Tr}} \bigl( Q^{-1} Y^2 \bigr)
$$
which is valid for non-negative symmetric matrices $Q, X,Y$.
Setting
$$
X = (D^2 \Phi)^{-1/2} D^2 \Phi_{e_i} (D^2 \Phi)^{-1/2}, \
Y =  (D^2 \Phi)^{-1/2} (V-W)^{(3)}_{i} (D^2 \Phi)^{-1/2},
$$
$$
Q = (D^2 \Phi)^{-1/2} (V+W)^{(2)} (D^2 \Phi)^{-1/2}
 $$
 one gets
\begin{align*}
2 \mbox{\rm{Tr}} & \Bigl[ (D^2 \Phi)^{-1} (V-W)^{(3)}_{i} (D^2 \Phi)^{-1} D^2 \Phi_{e_i}\Bigr]
\\&
\le \frac{1}{2} \mbox{\rm{Tr}} \Bigl[ \bigl[ (D^2 \Phi)^{-1/2} (V+W)^{(2)} (D^2 \Phi)^{-1/2} \bigr]^{-1}\bigl( (D^2 \Phi)^{-1/2} (V-W)^{(3)}_{i} (D^2 \Phi)^{-1/2} \bigr)^{2}\Bigr]
\\& +
2 \mbox{\rm{Tr}}\Bigl[  \bigl[ (D^2 \Phi)^{-1/2} (V+W)^{(2)} (D^2 \Phi)^{-1/2} \bigr]   \bigl( (D^2 \Phi)^{-1/2} D^2 \Phi_{e_i} (D^2 \Phi)^{-1/2}\bigr)^2 \Bigr].
\end{align*}

 \begin{figure}[h!]
\begin{center}

\begin{tikzpicture}[show background rectangle,>=stealth',shorten >=1pt,auto,node distance=1.3cm,
  thick,main node/.style={circle,draw,inner sep=1pt,minimum size=1pt}]

  \node[main node] (A) {$\Phi$};
  \draw (A) -- ++(0,-1.5);
  \draw[dotted] (A) -- ++(0.5,-1.5);
  \draw[dotted] (A) -- ++(-0.5,-1.5);
  \draw (A) -- ++(0.35,-1.05);
  \draw (A) -- ++(-0.35,-1.05);
  \node at ($(A) + (0, -2)$) {$X$};

  \node[main node] (B) [right= 1.5 of A]{$V-W$};
  \draw (B) -- ++(0,-1.5);
  \draw[dotted] (B) -- ++(0.5,-1.5);
  \draw[dotted] (B) -- ++(-0.5,-1.5);
  \draw (B) -- ++(0.35,-1.05);
  \draw (B) -- ++(-0.35,-1.05);
  \node at ($(B) + (0, -2)$) {$Y$};

  \node[main node] (C) [right= 1.5 of B]{$V+W$};
  \draw[dotted] (C) -- ++(0.5,-1.5);
  \draw[dotted] (C) -- ++(-0.5,-1.5);
  \draw (C) -- ++(0.35,-1.05);
  \draw (C) -- ++(-0.35,-1.05);
  \node at ($(C) + (0, -2)$) {$Q$};

  \node[main node] (D) [right=1.5 of C] {$\Phi$};
  \node[main node] (E) [right=1.2 of D] {$\Phi$};
  \draw (D) -- (E);
  \draw[dotted] (D) -- ++(-0.5,-1.5);
  \draw (D) -- ++(-0.35,-1.05);
  \draw (D) -- ++(0.5,-1.5);
  \draw[dotted] (E) -- ++(0.5,-1.5);
  \draw (E) -- ++(0.35,-1.05);
  \draw (E) -- ++(-0.5,-1.5);
  \node at ($(D) + (1, -2)$) {$X^2$};

\end{tikzpicture}

\end{center}
\caption{We let the reader guess the meaning of these operator-valued tensors.}
\end{figure}
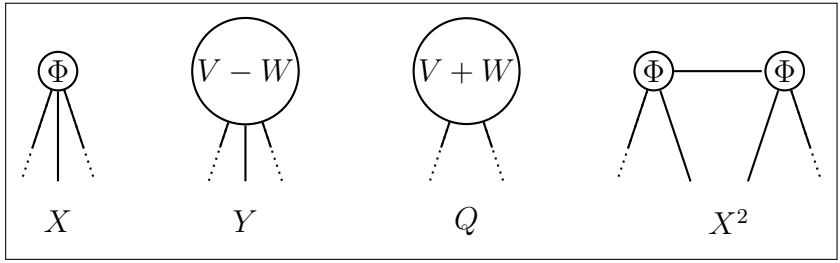

Taking into account (\ref{VW1}), (\ref{VW2}) we  rewrite this relation as follows:
$$
2 (V_{iab} - W_{iab}) \Phi_i^{ab} \le 2(V_{am} + W_{am}) \Phi^{m}_{ib} \Phi^{ab}_i + \frac{1}{2} H_{ab}.
$$
This readily implies the following inequality:
$$
 H_{ij} + 2 L g_{ij} \ge (V_{is} + W_{is}) g^{s}_j + (V_{js} + W_{js}) g^{s}_i+  g_{ki} g^{k}_{j}.
$$
The differential inequality for the corresponding norm $\|g\|$ can be obtained in the same way as in the proof of Theorem \ref{mainth}.
It remains to note that given the eigenvector $v$ of $g$ which corresponds to the largest eigenvalue $\lambda$
one has $H_{vv} \le \|H\|_{h}$, $\bigl( (V_{is} + W_{is}) g^{s}_j \bigr)(v,v) = \lambda (V_{vv} + W_{vv}) \ge 0$.
\end{proof}

\begin{lemma}
Assume that the measure $\mu$ has full support and  $V_{ij} + W_{ij} \ge 0$.
Assume, in addition,  that there exists $p>2$ such that  the Euclidean operator norm $\| (D^2 \Phi)^{-1}\|$  of $(D^2 \Phi)^{-1}$ belongs to $L^{p'}(\mu)$ with $p'>p$. Then
$$
\int  \|g\|^{p}_h \ d\mu \le \int  \|H\|^{\frac{p}{2}}_{h} d \mu.
$$
\end{lemma}
\begin{proof}
Set for brevity $\Lambda = \|g\|_h$.
Take a compactly supported smooth nonnegative function $\xi$. Multiply
(\ref{HLg}) by $\xi \Lambda^{p-2}$ and integrate over $\mu$.
$$
\int \|H\|_{h}  \Lambda^{p-2} \xi d \mu + 2 \int \xi L( \Lambda)  \Lambda^{p-2}  \ge \int \xi  \Lambda^p d \mu.
 $$
 Integrating by parts we get
 $$
 \int \|H\|_{h}  \Lambda^{p-2} \xi d \mu  - 2 \int \langle\nabla_h \xi, \nabla_h  \Lambda \rangle_h  \Lambda^{p-2}  d \mu \ge \int \xi  \Lambda^p d \mu
 + 2 (p-2) \int \xi  \Lambda^{p-3} \|\nabla_h  \Lambda \|^2 d \mu.
 $$
 Next
 $$
 - 2 \int \langle\nabla_h \xi, \nabla_h  \Lambda \rangle_h  \Lambda^{p-2}  d \mu
 \le 2 (p-2) \int \xi \|\nabla_h  \Lambda \|^2     \Lambda^{p-3}  d \mu + \frac{1}{2(p-2)} \int \frac{\|\nabla_h \xi\|^2_h}{\xi}  \Lambda^{p-1}  d \mu.
 $$
 This estimate yields
 $$
 \int \|H\|_{h}  \Lambda^{p-2} \xi d \mu  + \frac{1}{2(p-2)} \int \frac{\|\nabla_h \xi\|^2_h}{\xi}  \Lambda^{p-1}  d \mu  \ge \int \xi  \Lambda^p d \mu.
 $$
 For every $\varepsilon>0$ one gets by the H{\"o}lder inequality
 $$
 \frac{1}{2(p-2)} \int \frac{\|\nabla_h \xi\|^2_h}{\xi}  \Lambda^{p-1}  d \mu \le \int \varepsilon \Lambda^p \xi d\mu + c(\varepsilon,p) \int \Bigl| \frac{\nabla_h \xi}{\xi}\Bigr|_h^{2p} \xi d \mu.
 $$
 Finally, one has for every $\varepsilon>0$
 \begin{align*}
 \label{LHp}
 (1-\varepsilon) \int \xi  \Lambda^p d \mu & \le c(\varepsilon,p) \int \Bigl| \frac{\nabla_h \xi}{\xi}\Bigr|_h^{2p} \xi d \mu + \int \|H\|_{h}  \Lambda^{p-2} \xi d \mu
 \\&
 \le  c(\varepsilon,p) \int \Bigl| \frac{\nabla_h \xi}{\xi}\Bigr|_h^{2p} \xi d \mu +  \Bigl( \int \|H\|^{\frac{p}{2}}_{h} \xi d \mu\Bigr)^{\frac{2}{p}} \Bigl( \int  \Lambda^{p} \xi d \mu\Bigr)^{\frac{p-2}{p}}.
 \end{align*}
 It remains to show that there exists a sequence of non-negative smooth compactly supported functions $\{\xi_n\}$
 with the properties
 $$
 \xi_n \nearrow 1 \
 \mbox{\rm{pointwise}},\ \ \lim_n  \int \Bigl| \frac{\nabla_h \xi_n}{\xi_n}\Bigr|_h^{2p} \xi_n d \mu =0.
 $$
 Estimating
 $$
 | \nabla_h \xi_n|^2_h \le \| (D^2 \Phi)^{-1}\| |\nabla \xi_n|^2
 $$
 and applying H{\"o}lder inequality we get that it is enough to have
  $$
 \xi_n \nearrow 1  \
 \mbox{\rm{pointwise}}, \ \ \lim_n  \int \Bigl| \frac{\nabla \xi_n}{\xi_n}\Bigr|^{m} \xi_n d \mu =0
 $$
 for any (or sufficiently big) $m>2$. To
 solve this problem in dimension one
 we find a  non-negative compactly supported function
 $\eta$ with the properties
 $$
 \eta|_{\{x: |x| \le 1\}} =1, \ \eta(-x)=\eta(x), \ \sup_x \eta(x) \Bigl| \frac{\eta'(x)}{\eta(x)} \Bigr|^m < C_m < \infty, \ \forall m>0.
 $$
 The construction of such a function is standard.
 To obtain the desired sequence set $\xi^{(1)}_n(x)=1$ provided $|x| \le n$ and $\xi^{(1)}_n(x) = \eta(|x|-n+1)$ provided $|x| \ge n$.
 In the multidimensional case set $\xi_n(x) = \prod_{i=1}^n \xi^{(1)}_n(x_i)$.
\end{proof}

Letting $p$ to $\infty$ we get
\begin{corollary}
\label{gsqrtK}
Assume that the measures $\mu$ and $\nu$ are log-concave, $\mu$ has full support, and the Euclidean operator norm $\| (D^2 \Phi)^{-1}\|$ of $(D^2 \Phi)^{-1}$ is integrable
in any power.
Moreover, assume that
$$
\| H\|_{h} \le K
$$
for some constant $K>0$.
Then
$$
\|g\|_{h} \le \sqrt{K}.
$$
\end{corollary}

\begin{corollary}
\label{phi3}
Assume that  there exist  positive  constants $c, C, B$ such that
$$
c \le D^2 V \le C, \  c \le D^2 W \le C,
$$
$$
\sup_{e: |e|=1} \mbox{\rm{Tr}} \bigl[ D^2 V_{e} \bigr]^2 \le B, \ \sup_{e: |e|=1} \mbox{\rm{Tr}} \bigl[ D^2 W_e \bigr]^2 \le B
$$
(here $|\cdot|$ is the Euclidean norm).

Then there exists a constant $D(c, C)$ such that
$$
\sup_{e: |e|=1} \mbox{\rm{Tr}} \bigl[ D^2 \Phi_{e} \bigr]^2 \le D B.
$$
\end{corollary}
\begin{proof}
First we note that by the contraction theorem
\begin{equation}
\label{phi2}
\sqrt{\frac{c}{C}} \le D^2 \Phi \le \sqrt{\frac{C}{c}}.
\end{equation}
Thus the metric $h$ and the Euclidean metric are equivalent up to  factors depending on $c,C$.
In particular,
\begin{align*}
\sup_{e: |e|=1} \mbox{\rm{Tr}} \bigl[ D^2 \Phi_{e} \bigr]^2 &
\le \sup_{e: |e|=1}\frac{C}{c} \mbox{\rm{Tr}} \bigl[  D^2 \Phi^{-1} \cdot D^2 \Phi_{e} \bigr]^2
= \sup_{e: |e|=1} \frac{C}{c} g(e,e)
\\& \le
 \sup_{h: |u|_{h}=1}\Bigl(\frac{C}{c}\Bigr)^2 g(u,u) =
\Bigl(\frac{C}{c}\Bigr)^2 \|g\|_h.
\end{align*}
According to Corollary \ref{gsqrtK} one needs to estimate uniformly (in Euclidean or Riemannian norm) the matrix
 $$H_{ij} = \mbox{\rm{Tr}} \Bigl[ \bigl[  (V+W)^{(2)} \bigr]^{-1} (V-W)^{(3)}_{i} (D^2 \Phi)^{-1}(V-W)^{(3)}_{j} \Bigr]
= \langle X_i , X_j \rangle,$$
here  $ \langle A, B \rangle = \mbox{\rm{Tr}} (A^t B)$  and $$X_i=\bigl[(V+W)^{(2)} \bigr]^{-1/2} (V-W)^{(3)}_{i} (D^2 \Phi)^{-1/2}.
$$
It follows from (\ref{phi2}) and assumptions of the Theorem that
$$
c_1 \le (V+W)^{(2)} \le c_2,
$$
where $c_1, c_2$ depends on $c,C$.
Thus by the standard arguments
\begin{align*}
H_{ij} \le \langle X_i, X_j \rangle & \le \frac{C}{c} \big\langle [(V+W)^{(2)}]^{-1/2} (V-W)^{(3)}_{i} , [(V+W)^{(2)}]^{-1/2} (V-W)^{(3)}_{j}\big\rangle \\& \le \frac{C}{c c_1}
\langle (V-W)^{(3)}_{i}, (V-W)^{(3)}_{j} \rangle.
\end{align*}
Hence the Euclidean operator norm is controlled by the Euclidean operator norms of the following matrices
$$
A_{ij} = \big\langle D^2 V_{x_i} , D^2 V_{x_j} \big\rangle,
\ B_{ij} = \big\langle D^2 \Phi \cdot D^2 W_{x_i} \cdot D^2 \Phi,   D^2 \Phi \cdot D^2 W_{x_j} \cdot D^2 \Phi\big\rangle.
$$
The desired estimate follows immediately from the assumptions of the Theorem
and (\ref{phi2}).
\end{proof}

\begin{corollary}
\label{phi3Gauss}
Assume that
$D^2 V >c >0$
 and $\nu = c_{Q} e^{-Q(x,x)} dx$ is Gaussian. Then
 $$
 \| g\|^2_{h}  \le \sup_{x \in \mathbb{R}^n} \| H(x)\|_{h} \le \frac{\|Q\|}{c} \sup_{x, e \in \mathbb{R}^n: |e|=1}  \mbox{\rm{Tr}} \Bigl[  (D^2 V)^{-1} \bigl( D^2 V_{e}\bigr)^2 \Bigr](x)
 $$
 provided the right-hand side is finite.
\end{corollary}
\begin{proof}
According to the contraction theorem $D^2 \Phi \ge \sqrt{\frac{c}{\|Q\|}}$.
In particular, this implies
$$
\| H\|_{h}  = \sup_{v:  \ \langle D^2 \Phi v, v\rangle \le 1} H_{vv} \le \sup_{v: \  |v|^2 \le \sqrt{\frac{\|Q\|}{c}} } H_{vv}
\le \sqrt{\frac{\|Q\|}{c}} \| H\|.
$$
One has $W_{ijk}=0, W_{ij} \ge 0$. This implies
$$
H_{ij} \le  \sqrt{\frac{\|Q\|}{c}} \mbox{\rm{Tr}} \Bigl[  (D^2 V)^{-1} \cdot D^2 V_{e_i} \cdot D^2 V_{e_j} \Bigr].
$$
and the claim follows.
\end{proof}

\begin{remark}
The results of Corollaries \ref{phi3}, \ref{phi3Gauss} are dimension-free and have natural analogues for infinite-dimensional measures. For instance, some natural estimates of this type holds for the potential $\varphi$
of the optimal transporation $T(x) =  x + \nabla \varphi(x)$ pushing forward
$g \cdot \gamma$ onto a (infinite dimensional) Gaussian measure $\gamma$, where
$\nabla \varphi$ is understood as a gradient along the  Cameron-Martin space of $\gamma$ (see \cite{Kol2004}, \cite{BoKo2012}, \cite{BoKo2013} and the references therein).
\end{remark}


\begin{thebibliography}{10}


\bibitem{BQ}
 Bakry~D.,  Qian~Z., Volume comparison theorems without Jacobi fields, Current trends in potential theory,
Theta Ser. Adv. Math., vol. 4, Theta, Bucharest, 2005, pp. 115--122.

\bibitem{BB} Berman~R. J., Berndtsson~ B.,  Real Monge-Amp\`ere equations and K\"ahler-Ricci solitons on toric log Fano varieties.
Ann. Fac. Sci. Toulouse Math. (6) 22 (2013), no. 4, 649--711.


\bibitem{BoKo2012}
Bogachev~V.I., Kolesnikov~A.V., The Monge--Kantorovich problem: achievements, connections, and perspectives,
Russian Mathematical Surveys, 67(5), 2012, 785--890.

\bibitem{BoKo2013}
Bogachev~V.I., Kolesnikov~A.V., Sobolev regularity for the Monge--Ampere equation in the Wiener space. Kyoto J. Math. Volume 53, Number 4 (2013), 713--738.

\bibitem{Calabi1}
Calabi~E.,  Improper affine hyperspheres of convex type and a generalization of a theorem by K. J{\"o}rgens. Michigan Math. J. 5 (1958), no. 2, 105--126.

\bibitem{Calabi2}
Calabi~E.,
Complete affine hyperspheres. I, Symposia Mathematica, Vol. X (Convegno di Geometria Differenziale,
INDAM, Rome, 1971), Academic Press, London, 1972, pp. 19--38.

\bibitem{CoKl}
Cordero-Erausquin~D.,  Klartag~B.,  Moment measures. J. Functional Analysis,  268 (12), (2015), 3834--3866.

\bibitem{CY1}
Cheng~S.-Y., Yau S.-T., On the existence of a complete K{\"a}hler metric on noncompact complex manifolds and the regularity of
FeffermanпїЅs equation, Comm. Pure Appl. Math. 33 (1980), no. 4, 507--544.

\bibitem{CY2}
Cheng~S.-Y., Yau S.-T., The real Monge-Amp\'ere equation and affine flat structures, Proceedings of the 1980 Beijing Symposium
on Differential Geometry and Differential Equations, Vol. 1, 2, 3 (Beijing, 1980) (Beijing), Science Press, 1982,
pp. 339--370.

\bibitem{CY3}
Cheng~S.-Y., Yau S.-T., Complete affine hypersurfaces. I. The completeness of affine metrics, Comm. Pure Appl. Math. 39
(1986), no. 6, 839--866.


\bibitem{CK}
 Chow~B., Knopf~D. The Ricci Flow: An Introduction
Mathematical surveys and monographs, AMS, 2004.


\bibitem{Fox}
Fox~D.J.F., A Schwarz lemma for K{\"a}hler affine metric and the canonical potential of a proper convex cone.,
Annali di Matematica Pura ed Applicata (2015),
194 (1),  1--42.


\bibitem{K_moment} Klartag~ B., { Poincar\'e inequalities and moment maps. }
Ann. Fac. Sci. Toulouse Math., Vol. 22, No. 1, (2013), 1--41.


\bibitem{K_part_I} Klartag~B., { Logarithmically-concave moment measures I}.
Geometric Aspects of Functional Analysis, Lecture Notes in Math. 2116 , Springer (2014), 231--260.

\bibitem{KlarKol}
Klartag~B.B., Kolesnikov~A.V.,
Eigenvalue distribution  of optimal transportation. Anal. \& PDE's. 2015, 8 (1),  33--55.

\bibitem{Kol2004} Kolesnikov~A.V.
Convexity inequalities and optimal transport of infinite-dimensional measures.
J. Math. Pures Appl., 2004, v. 83, n. 11, p 1373--1404.


\bibitem{Kol2010}
Kolesnikov~A.V., On Sobolev regularity of mass transport and transportation inequalities,
Theory Probab. Appl., (2013), 57(2), 243--264.

\bibitem{Kol2011}
Kolesnikov~A.V., Mass transportation and contractions, MIPT Proc. (2010), 2(4), 90-99. arxiv: 1103.1479

\bibitem{kol}  Kolesnikov~A.V., {Hessian metrics, CD(K,N)-spaces, and optimal transportation of log-concave measures}.
 Discrete and Continuous Dynamical Systems - Series A., Vol. 34, No. 4, (2014), 1511--1532.

\bibitem{KolMil}
Kolesnikov~A.V.,  Milman~E., Poincar{\'e} and Brunn-Minkowski inequalities on weighted Riemannian manifolds with boundary.
arXiv:1310.2526.

\bibitem{Lof} Loftin~J., Affine spheres and K\"ahler-Einstein metrics.
Math. Res. Lett., Vol. 9, no. 4, (2002),  425--432.

\bibitem{Oss} Osserman~R., On the inequality $\Delta u \geq f(u)$.
Pacific J. Math., Vol. 7, (1957), 1641--1647.

\bibitem{Sas} Sasaki~T., A note on characteristic functions and projectively invariant metrics on a bounded convex domain.
Tokyo J. Math., Vol. 8, no. 1, (1985), 49--Ѕ79.

\bibitem{Vill}
Villani~C.,
 Topics in optimal transportation,
Amer. Math. Soc. Providence, Rhode Island, 2003.

\bibitem{WZ} Wang~X.-J., Zhu~X.,
K\"ahler--Ricci solitons on toric manifolds with positive first Chern
class. Advances in Math., Vol. 188, (2004), 87--103.



\end{thebibliography}
\end{document}